\documentclass[11pt,hidelinks]{article}

\usepackage{amsmath}
\usepackage{amssymb}
\usepackage{amsthm}
\usepackage{amsfonts}
\usepackage{bbm}
\usepackage{bm}
\usepackage{multirow}
\usepackage{natbib}
\usepackage{graphicx}
\usepackage{dsfont}
\usepackage{mathrsfs}
\usepackage{xcolor}
\usepackage{microtype} 
\usepackage[shortlabels]{enumitem}
\usepackage{hyperref}
\usepackage[capitalise,noabbrev]{cleveref}
\hypersetup{colorlinks = true, allcolors = blue}
\usepackage{authblk}
\usepackage{mathtools}
\usepackage{mleftright}
\usepackage[margin=1in]{geometry}
\usepackage{tikz}
\usepackage{pgfplots}
\pgfplotsset{compat=1.18}
\usepgfplotslibrary{groupplots}

\newlist{coroitems}{enumerate}{1}
\setlist[coroitems,1]{%
  label=\textup{\thecoro\alph*.},
  ref=\textup{\thecoro\alph*}
}

\newlist{propitems}{enumerate}{1}
\setlist[propitems,1]{%
  label=\textup{\theprop\alph*.},
  ref=\textup{\theprop\alph*}
}

\newlist{thmitems}{enumerate}{1}
\setlist[thmitems,1]{%
  label=\textup{\thethm\alph*.},
  ref=\textup{\thethm\alph*}
}

\crefname{coroitemsi}{Corollary}{Corollaries}
\Crefname{coroitemsi}{Corollary}{Corollaries}

\crefname{propitemsi}{Proposition}{Propositions}
\Crefname{propitemsi}{Proposition}{Propositions}

\crefname{thmitemsi}{Theorem}{Theorems}
\Crefname{thmitemsi}{Theorem}{Theorems}

\newcommand{\eps}{\varepsilon}

\newcommand{\bepsilon}{\bm{\varepsilon}}
\newcommand{\bzero}{\bm{0}}
\newcommand{\bone}{\bm{1}}
\newcommand{\Ab}{\mathbf{A}}
\newcommand{\Vb}{\mathbf{V}}
\newcommand{\bb}{\bm{b}}
\newcommand{\bn}{\bm{n}}

\newcommand{\br}{\bm{r}}
\newcommand{\bs}{\bm{s}}
\newcommand{\bv}{\bm{v}}
\newcommand{\bw}{\bm{w}}
\newcommand{\bx}{\bm{x}}
\newcommand{\Ind}[1]{ \mathds{1} \left\{ #1 \right\} }
\newcommand{\C}{\mathbb{C}}
\newcommand{\N}{\mathbb{N}}
\newcommand{\R}{\mathbb{R}}
\usepackage{calrsfs}
\DeclareMathAlphabet{\pazocal}{OMS}{zplm}{m}{n}

\newcommand{\cB}{\pazocal{B}}
\newcommand{\cF}{\pazocal{F}}
\newcommand{\cP}{\pazocal{P}}
\newcommand{\cS}{\pazocal{S}}
\newcommand{\cX}{\pazocal{X}}

\newcommand{\E}{\mathbb{E}}
\newcommand{\EE}{\mathbb{E}}
\newcommand{\Prob}{\mathbb{P}}

\newcommand{\stdunif}{\mathrm{Unif}(0,1)}
\newcommand{\cPab}{\cP_{[a,b]}}
\newcommand{\cAab}{\cS^{\mathrm{adm}}_{[a,b]}}
\newcommand{\cFmeps}{\cF_m(c_{1:m})}

\renewcommand{\left}{\mleft}
\renewcommand{\right}{\mright}

\newcommand{\fallfact}[1]{^{\underline{#1}}}
\newcommand{\anti}[1]{^{( #1 )}}

\newcommand{\epsseq}{\{\eps_n\}_n}
\newcommand{\wc}{\rightsquigarrow}

\theoremstyle{plain}
\newtheorem{thm}{Theorem}
\newtheorem{lemm}{Lemma}
\newtheorem{prop}{Proposition}

\newtheorem{coro}{Corollary}

\theoremstyle{definition}
\newtheorem{ex}{Example}

\theoremstyle{remark}
\newtheorem*{rem}{Remark}

\crefname{thm}{Theorem}{Theorems}
\crefname{lemm}{Lemma}{Lemmas}
\crefname{prop}{Proposition}{Propositions}
\crefname{conj}{Conjecture}{Conjectures}
\crefname{coro}{Corollary}{Corollaries}
\crefname{ex}{Example}{Examples}
\crefname{con}{Condition}{Conditions}
\crefname{assum}{Assumption}{Assumptions}
\crefname{algo}{Algorithm}{Algorithm}
\crefname{defi}{Definition}{Definitions}

\newcommand{\dif}{\mathop{}\!\mathrm{d}}

\title{Primitive Sequences for Probability Measures on Compact Intervals}

\author{Robert Zimmerman\thanks{Email: \texttt{\url{robert.zimmerman@alumni.utoronto.ca}}}}

\date{\today}

\begin{document}

\maketitle

\begin{abstract}
    We introduce a sequence representation of a random variable $X$ supported on a compact interval $[a,b]$, which we call a \textit{primitive sequence}. We construct this sequence by repeatedly antidifferentiating the associated cumulative distribution function of $X$ and evaluating the antiderivatives at the endpoint $b$. We show that the primitive sequence of $X$ can be identified as a factorially rescaled moment sequence of the reflected random variable $b-X$. Through this identification, we show that the primitive sequence transparently captures qualitative features of the distribution of $X$. We then connect primitive sequences directly to classical moment theory and exploit this connection to characterize admissible primitive sequences and to show that under natural topologies, the map from probability measures to primitive sequences is a homeomorphism. We end by examining the set of probability measures whose first $m$ primitive sequence terms are fixed, and thereby obtaining sharp upper and lower bounds on two functionals of those measures.
\end{abstract}

\section{Introduction}

A probability measure supported on a compact interval may be identified in many equivalent ways: by its cumulative distribution function, by its moments, or by its quantile function, to name only a few. In this paper, we study a sequence representation that we call a \textit{primitive sequence}, which is obtained directly from the cumulative distribution function by repeated antidifferentiation. More precisely, if a random variable $X$ is supported on a compact interval $[a,b]$ with cdf $F_X$, we consider the endpoint values
\[
F_X\anti{-1}(b), F_X\anti{-2}(b), \ldots,
\]
where the repeated antiderivatives are specified by the boundary conditions
\[
F_X\anti{-n}(a)=0,\qquad n\ge 1,
\]
and the resulting sequence $\{F_X\anti{n}(b)\}_n$ is the primitive sequence associated with the distribution of $X$.

We prove in \cref{sec:reptheorem} that primitive sequences are exactly reflected moment sequences with factorial scaling. Thus, while primitive sequences are not separate from the standard theory of moment sequences, they allow us to recast the classical Hausdorff moment problem \citep{hausdorff1921summationsmethoden} and its consequences into an alternative coordinate system which is transparently derived from the distribution function and encodes qualitative features of the distribution in a more transparent way; for example, mass near the lower support endpoint $a$, atomic components, and local smoothness of the density (when the distribution is absolutely continuous) are all reflected in the decay of the corresponding primitive sequence.

There is some precedent for the use of iterated integral transforms of distribution functions and tails to characterize distributional behaviour. In stochastic ordering, for example, \cite{rolskiOrderRelationsSet1976} uses integrated distribution functions and survival functions to define and analyze order relations. More broadly, distributions have been represented via integral functionals of the cdf or quantile function and related coordinate systems including $L$-statistics \citep{chernoffAsymptoticDistributionLinear1967}, and similar transforms also appear in approximation and nonparametric estimation via approximation of densities by growing exponential family models \citep{barronApproximationDensityFunctions1991} and via orthogonal series and wavelets \citep[][Chapters~2--3]{efromovichNonparametricCurveEstimation1999}. Related objects arise in actuarial science through the $n$th \emph{stop-loss transform} (or $n$th partial moment) $\E[(X-u)^n_+]$, used in comparing insurance risks and in ruin-related analyses \citep{chengNthStoplossTransform2003,riffiStoplossTransformProperties2003}, and through the $n$th \emph{lower partial moment} $\E[(u - X)_+^n]$, used as a measure of downside risk \citep{bawaOptimalRulesOrdering1975,fishburnMeanRiskAnalysisRisk1977}. However, these works develop properties of $u \mapsto \E[(X-u)^n_+]$ or $u \mapsto \E[(u - X)_+^n]$ as \emph{functions} of the threshold $u$. Our focus is different: we define the primitive sequence by iterating antiderivatives of the cdf subject to fixed boundary conditions and then recording the resulting endpoint value at $b$. To our knowledge, this particular moment sequence has not been previously studied.

The remainder of the paper is organized as follows. \cref{sec:reptheorem} develops the basic representation theorem and elementary properties of primitive sequences. \cref{sec:examples} presents several examples chosen to illustrate how primitive sequences capture important qualitative features of their underlying laws. \cref{sec:admissiblesequences} characterizes admissible primitive sequences via the Hausdorff criterion and studies the mapping from measures to primitive sequences from a topological point of view, showing that the mapping is affine, injective, and continuous with respect to the weak topology on probability measures and the product topology on primitive sequences, and is in fact a homeomorphism onto its image with respect to these topologies. \cref{sec:finitetruncation} proves the convexity of the set of probability measures whose first $m$ primitive sequence terms are fixed, and then studies two optimization problems which produce sharp upper and lower bounds on certain functionals of the measures. \cref{sec:discussion} concludes with a brief discussion and suggestions for future work.

\subsection{Setup and notation}

Throughout, we fix a probability space $(\Omega, \cF, \Prob)$ and a real-valued random variable $X$ supported on $[a,b]$, where $-\infty < a < b < \infty$; that is, $\Prob(X \in [a,b]) = 1$ (continuity or full support is \emph{not} assumed). The probability measure associated with $X$ is $\mu$, and we write $F_X(x) := \Prob(X \leq x) = \mu((-\infty,x])$ for the cdf of $X$, defined on $\R$; we occasionally write $F_\mu(x) := \mu((-\infty, x])$ when an explicit reference to a random variable is unnecessary. The $L^\infty$-norm (or uniform norm or essential supremum) of $X$ is $\|X\|_\infty := \inf\{M \geq 0: |X| \leq M \text{ a.s.}\}$. The left limit at $x^*$ of a function $f:\R \to \R$ is written as $f(x^*-) := \lim_{x \nearrow x^*} f(x)$. $C([a,b])$ is the set of continuous functions from $[a,b]$ to $\R$, $\cPab$ denotes the set of Borel probability measures on $[a,b]$, and $\cB_{[a,b]}$ is the set of Borel sets of $[a,b]$. We write $\mu_k \wc \mu$ to mean that the sequence $\{\mu_k\}_k$ converges weakly to $\mu$ as $k \to \infty$. Infinite sequences $\{a_1, a_2, \ldots\}$ are written compactly as $\{a_n\}_n$, and we also write $a_{1:m} := \{a_1, \ldots, a_m\}$ for $m \geq 1$. Finally, the set of extreme points of a convex set $A$ is defined as
\[
    \mathrm{ext}(A) := \{x \in A: \text{ there do not exist } x_1 \neq x_2 \in A \text{ and } \lambda \in (0,1) \text{ such that } x = \lambda x_1 + (1-\lambda) x_2 \}.
\]

\section{Representation theorem and consequences}\label{sec:reptheorem}

We begin by formally defining a primitive sequence. For $n \geq 0$, define the $n$-fold antiderivative $F_X\anti{-n}:
\R \to \R$ of $F_X$ by $F_X\anti{0} := F_X$ and recursively by
\begin{equation*}
    F_X\anti{-n}(t) := \int_a^t F_X\anti{-(n-1)}(x) \dif x, \qquad n \geq 1,
\end{equation*}
It follows from basic properties of cdfs that for each $n \geq 1$, $F_X\anti{-n}$ is continuous on $\R$, non-negative and non-decreasing on $[a,\infty)$, and satisfies the boundary condition $F_X\anti{-n}(a) = 0$ automatically. Our central object of study is the sequence of values of $F_X\anti{-n}$ evaluated at $b$ for each $n \geq 0$; thus, we define the sequence $\epsseq$ where
\begin{equation}\label{eq:epsseqdef}
    \eps_n = \eps_n(\mu) := F_X\anti{-n}(b), \qquad n \geq 0.
\end{equation}
We call $\epsseq$ the \emph{primitive sequence} associated with $\mu$. In this section, we establish basic properties of $\epsseq$, relate it explicitly to the raw moments of $X$, and identify its important monotonicity and convexity features.

Our first result is a fundamental identity that we will use repeatedly throughout the paper. While \cref{eq:epsseqdef} defines $\epsseq$ in terms of iterated integration of $F_X$, we show here that $\eps_n$ is in fact simply a scaled moment of the distance from $X$ to the upper support endpoint $b$:
\begin{thm}\label{thm:representation} We have
\begin{equation}\label{eq:basicrep}
    \eps_n = \frac{1}{n!} \E[(b-X)^n], \qquad n \geq 0.
\end{equation}
\end{thm}

\begin{proof}
    The result is trivial when $n = 0$, so we may assume that $n \geq 1$. According to Cauchy's formula for repeated integration \citep[][Leçon 35]{cauchyResumeLeconsDonnees2009}, for $t \in [a,b]$ we have
    \begin{equation*}
        F_X\anti{-n}(t) = \frac{1}{(n-1)!} \int_a^t (t-x)^{n-1} F_X(x) \dif x,
    \end{equation*}
    and with the choice $t = b$ we obtain
    \begin{equation}\label{eq:epsCauchy}
        \eps_n = \frac{1}{(n-1)!} \int_a^b (b-x)^{n-1} F_X(x) \dif x.
    \end{equation}
    Now, 
    \begin{equation*}
        (b-X)^n = n \int_X^b (b-x)^{n-1} \dif x
        = n \int_a^b (b-x)^{n-1} \Ind{X \leq x} \dif x.
    \end{equation*}
    Taking expectations gives
    \begin{equation}\label{eq:EbmXn}
        \E[(b-X)^n] = n \E\left[\int_a^b (b-x)^{n-1} \Ind{X \leq x} \dif x \right]
        = n \int_a^b (b-x)^{n-1} F_X(x) \dif x,
    \end{equation}
    where we have used Tonelli's theorem to interchange the expectation with the integral since the integrand is non-negative. Combining \cref{eq:epsCauchy,eq:EbmXn} gives 
    \begin{equation*}
        \eps_n = \frac{1}{(n-1)!} \frac{\E[(b-X)^n]}{n}
        = \frac{1}{n!} \E[(b-X)^n],
    \end{equation*}
    as desired.
\end{proof}

\begin{rem}
    The so-called \emph{Darth Vader rule} \citep{muldowney2012darth} for non-negative random variables is that if $Y \geq 0$ a.s., then 
    \begin{equation}\label{eq:DARTHVADER}
        \E[Y] = \int_0^\infty \left(1 - F_Y(x)\right) \dif x.
    \end{equation}
    The rule can be generalized to random variables $X \in [a,b]$ a.s. (possibly with $\Prob(X < 0) > 0$) by writing $X = b - (b-X)$. Since $b-X \geq 0$ a.s., \cref{eq:DARTHVADER} yields
    \begin{equation}\label{eq:DARTHVADER2}
        \E[X] = b - \int_0^\infty \left(1 - F_{b-X}(x)\right) \dif x
        = b - \int_a^b F_X(x) \dif x,
    \end{equation}
    where the second equality holds because $F_{b-X}(x) = 1$ for $x \geq b-a$. \cref{eq:DARTHVADER2} is a special case of \cref{thm:representation}: simply take $n = 1$ in \cref{eq:basicrep} to get the equivalent statement that $\E[b-X] = \eps_1 = F_X\anti{-1}(b)$. Thus, \cref{thm:representation} can be viewed as a higher-order Darth Vader rule, expressing all moments of $b-X$ in terms of antiderivatives. (For other generalizations of the Darth Vader rule, see \cite{muldowney2012darth}.)
\end{rem}

From \cref{thm:representation}, it is clear that the primitive sequence contains complete information about the raw moments of $X$, in the sense that $\E[X], \E[X^2], \ldots, \E[X^n]$, can be extracted from $\eps_1, \eps_2, \ldots, \eps_n$ for any $n \geq 0$ through a finite sequence of linear operations that do not depend on $\mu$. This follows simply from linearity of expectation and the fact that the binomial expansion of the degree-$n$ polynomial $(b-x)^n$ includes all powers of $x$ up to and including the $n$th. Still, it will be convenient to record explicit expressions for the $n$th moment of $X$ in terms of the first $n$ terms of the primitive sequence, and vice versa:

\begin{coro}\label{cor:momentstoeps}
    We have
    \begin{equation}\label{eq:momentstoeps}
        \E[X^n] = \sum_{j=0}^n (-1)^j n\fallfact{j} b^{n-j} \eps_j \quad \text{and} \quad
        \eps_n = \sum_{j=0}^n \frac{(-1)^jb^{n-j}}{j! (n-j)!} \E[X^j],
    \end{equation}
    where $n\fallfact{j} := n (n-1) (n-2) \cdots (n-j+1)$ is the $j$th falling factorial of $n$.
\end{coro}
\begin{proof}
    Using the binomial theorem and \cref{thm:representation}, we have
    \begin{align*}
        \E[X^n] &= \E[((-1)(b-X) + b)^n] \\
        &= \sum_{j=0}^n (-1)^j\binom{n}{j} b^{n-j} \E[(b-X)^j] \\
        &= \sum_{j=0}^n (-1)^j\frac{n!}{(n-j)!} b^{n-j} \eps_j \\
        &= \sum_{j=0}^n (-1)^j n\fallfact{j} b^{n-j} \eps_j,
    \end{align*}
    which is the first expression in \cref{eq:momentstoeps}. For the second expression, the binomial theorem similarly yields
    \begin{equation*}
        \eps_n = \frac{1}{n!}\E[(b-X)^n] 
        = \frac{1}{n!} \sum_{j=0}^n \binom{n}{j} (-1)^j b^{n-j} \E[X^j]
        = \sum_{j=0}^n \frac{(-1)^j b^{n-j}}{j! (n-j)!} \E[X^j],
    \end{equation*}
    as desired.
\end{proof}

\begin{rem}
    An alternative proof of \cref{cor:momentstoeps} sets up the infinite linear system $\bm{m} - \bb^\bullet = \Ab \bepsilon$ where $\bm{m} := (\EE[X], \EE[X^2], \E[X^3], \cdots)^\top$, $\bepsilon := (\eps_1, \eps_2, \eps_3, \ldots)^\top$, $\bb^\bullet := (b, b^2, b^3, \ldots)^\top$ and $\Ab$ is the infinite lower triangular matrix
    \[
        \Ab := \begin{pmatrix}
            -1 & 0 & 0 & 0 & \cdots \\
            -2b & 2 \cdot 1 & 0 & 0 & \cdots \\
            -3b^2 & 3 \cdot 2b & -3 \cdot 2 \cdot 1 & 0 & \cdots \\
            -4b^3 & 4 \cdot 3b^2 & -4\cdot3\cdot2b & 4 \cdot 3 \cdot 2 \cdot 1 & \cdots \\
            \vdots & \vdots & \vdots & \vdots & \ddots\\
        \end{pmatrix}.
    \]
    For each $m \geq 1$, the $m \times m$ upper-left submatrix of $\Ab$ is lower triangular with diagonal entries $-1!, 2!, \ldots, (-1)^m m!$ and hence invertible, and the first $m$ terms of the primitive sequence $\eps_1, \ldots, \eps_m$ can therefore be obtained uniquely from $\EE[X], \ldots, \EE[X^m]$ via back-substitution. While this route to proving \cref{cor:momentstoeps} is clearly more involved than our straightforward application of the binomial theorem, it does show explicitly that passing between the raw moment sequence $\{\E[X^n]\}_n$ and the primitive sequence $\epsseq$ is equivalent to a change of basis in $\R^m$ for each $m \geq 1$.
\end{rem}

Next, we collect several basic properties of the primitive sequence $\epsseq$ which we will use throughout the paper. Each of these properties is effectively a consequence of the compact support of $X$:
\begin{coro}\label{coro:basicproperties}
    The following hold:
    \begin{coroitems}
        \item\label{cor:sequencebounds} $0 \leq \eps_n \leq (b-a)^n/n!$. Consequently, $\eps_n \to 0$ and $\epsseq \in \ell^1(\N)$. 
        \item\label{cor:seqboundsequality} $\eps_n = (b-a)^n/n!$ for all $n \geq 0$ if and only if $X = a$ a.s., and $\eps_n = 0$ for all $n \geq 1$ if and only if $X = b$ a.s.
        \item\label{cor:log-convex} The sequence $\{n!\eps_n\}_n$ is log-convex, and hence
        \begin{equation}\label{eq:almostlogcon}
            \eps_n^2 \leq \left(1 + \frac{1}{n}\right) \eps_{n-1} \eps_{n+1}.
        \end{equation}
    \end{coroitems}
\end{coro}
\begin{proof}
    Because $X \in [a,b]$ a.s., we have $b-X \in [0,b-a]$ a.s. and so $(b-X)^n \in [0, (b-a)^n]$ a.s. Therefore $0 \leq \E[(b-X)^n] \leq (b-a)^n$; upon dividing through by $n!$ and applying \cref{thm:representation}, we obtain $0 \leq \eps_n \leq (b-a)^n/n!$, which immediately implies that $\eps_n \to 0$ as $n \to \infty$. Moreover,
    \begin{equation*}
        \sum_{n=0}^\infty \eps_n
        \leq \sum_{n=0}^\infty \frac{(b-a)^n}{n!}
        = e^{b-a}
        < \infty,
    \end{equation*}
    which proves \cref{cor:sequencebounds}.

    If $X = a$ a.s., then \cref{thm:representation} immediately yields $\eps_n = \E[(b-a)^n]/n! = (b-a)^n/n!$ for all $n \geq 0$. Conversely, if $\eps_n = (b-a)^n/n!$ for all $n \geq 0$, then $\E[(b-X)^n] = (b-a)^n$ and so $\E[(b-a)^n - (b-X)^n] = 0$. But $(b-a)^n - (b-X)^n \geq 0$ a.s., and therefore $(b-a)^n - (b-X)^n = 0$ a.s., whence $X = a$ a.s. On the other hand, if $X = b$ a.s., then \cref{thm:representation} yields $\eps_n = 0$ for all $n \geq 1$. Conversely, if $\eps_n = 0$ for all $n \geq 1$, then $\E[b-X] = n! \eps_1 = 0$, and because $b - X \geq 0$ a.s., we must have $X = b$ a.s., which proves \cref{cor:seqboundsequality}.

    Finally, the Cauchy-Schwarz inequality yields
    \begin{equation}\label{eq:CSineq}
        \E\left[(b-X)^n \right]^2 = \left(\E\left[(b-X)^\frac{n-1}{2} (b-X)^\frac{n+1}{2} \right] \right)^2
        \leq \E\left[(b-X)^{n-1}\right] \E\left[(b-X)^{n+1}\right].
    \end{equation}
    Substituting in $\E[(b-X)^{n-1}] = (n-1)! \eps_{n-1}$ by \cref{thm:representation} and $\E[(b-X)^{n+1}] = (n+1)!\eps_{n+1}$ similarly, \cref{eq:CSineq} can be written as 
    \[
        (n!\eps_n)^2 \leq ((n-1)!\eps_{n-1}) ((n+1)!\eps_{n+1}),
    \]
    which shows that $\{n! \eps_n\}_n$ is log-convex. Dividing through by $(n!)^2$ and rearranging yields \cref{eq:almostlogcon}, proving \cref{cor:log-convex}.
\end{proof}

The next corollary demonstrates several direct connections between a primitive sequence and its associated random variable $X$. It is convenient to normalize the primitive sequence in order to remove its dependence on the length of the interval $[a,b]$. To this end, we introduce a standardized version of the primitive sequence $\epsseq$, allowing for a straightforward analysis of its limiting behavior in terms of the mass $X$ places at the lower endpoint $a$:

\begin{coro}\label{coro:normalizedsequence}
    Define the normalized sequence $\{\gamma_n\}_n$ by
    \begin{equation*}
        \gamma_n := \frac{n! \eps_n}{(b-a)^n}
        =
        \E\left[\left(\frac{b-X}{b-a} \right)^n \right] \in [0,1], \qquad n \geq 0.
    \end{equation*}
    Then the following hold:
    \begin{coroitems}
        \item\label{cor:lefttailID} The sequence $\{\gamma_n\}_n$ is non-increasing and bounded in $[0,1]$, and satisfies $\gamma_n \searrow \Prob(X = a)$ as $n \to \infty$. 
        \item\label{cor:seqsharperbound} For all $\delta \in (a,b)$, 
        \[
            \gamma_n
            \leq \Prob(X \leq \delta) + \left(\frac{b- \delta}{b-a}\right)^n, \qquad n \geq 0.
        \]
        \item\label{cor:gapexpdecay} If there exists $\delta > 0$ such that $\Prob(X \geq a + \delta) = 1$, then 
        \begin{equation}\label{eq:gapexpdecay}
            \eps_n \leq \frac{(b-a-\delta)^n}{n!}, \qquad n \geq 0.
        \end{equation}
        \item\label{cor:tiltedmean} If $\Prob(X = a) > 0$, then
        \begin{equation}\label{eq:ratiolimit}
            \lim_{n \to \infty} \frac{(n+1)\eps_{n+1}}{\eps_n} = b-a.
        \end{equation}
    \end{coroitems}
\end{coro}
\begin{proof}
    To begin with, observe that $(b - X)/(b-a) \in [0,1]$ a.s., and therefore 
    \begin{equation*}
        0 \leq \left(\frac{b - X}{b-a}\right)^{n+1} 
        \leq \left(\frac{b - X}{b-a}\right)^n 
        \leq 1 \text{ a.s.}
    \end{equation*}
    By monotonicity of expectation, it follows that
    \begin{equation*}
        0 \leq \E\left[\left(\frac{b - X}{b-a}\right)^{n+1} \right] 
        \leq \E\left[\left(\frac{b - X}{b-a}\right)^n \right] 
        \leq 1, 
    \end{equation*}
    which, using \cref{thm:representation}, we can rewrite as
    \[
        0 \leq \frac{(n+1)! \eps_{n+1}}{(b-a)^{n+1}} \leq \frac{n! \eps_n}{(b-a)^n} \leq 1.
    \]
    That is, $0 \leq \gamma_{n+1} \leq \gamma_n \leq 1$, and hence $\{\gamma_n\}_n$ is non-increasing and bounded in $[0,1]$. To determine the limit of this sequence, we write 
    \begin{align*}
        \gamma_n &= \E\left[\left(\frac{b-X}{b-a}\right)^n \right] \\
        &= \E\left[\left(\frac{b-X}{b-a}\right)^n \Ind{X = a}\right] + \E\left[\left(\frac{b-X}{b-a}\right)^n \Ind{X > a}\right] \\
        &= \Prob(X = a) + \E\left[\left(\frac{b-X}{b-a}\right)^n \Ind{X > a}\right].
    \end{align*}
    On the event $X > a$, we have $0 \leq (b-X)/(b-a) \in [0,1)$, and hence
    \begin{equation*}
        \left(\frac{b - X}{b-a}\right)^n \Ind{X > a} \longrightarrow 0 \text{ a.s.}
    \end{equation*}
    as $n \to \infty$. Moreover, since $((b-X)/(b-a))^n \Ind{X>a} < 1$ a.s., the bounded convergence theorem applies and we obtain
    \[
        \E\left[\left(\frac{b-X}{b-a}\right)^n \Ind{X > a}\right] \longrightarrow 0
    \]
    as $n \to \infty$ and hence $\gamma_n \to \Prob(X = a)$ as $n \to \infty$, which proves \cref{cor:lefttailID}.

    Similarly, for any $n \geq 0$ we may write  
    \begin{equation}\label{eq:EbmXsplit}
        \E\left[(b-X)^n\right] = \E\left[(b-X)^n \Ind{X \leq \delta}\right] + \E\left[(b-X)^n \Ind{X > \delta}\right].
    \end{equation}
    On the event $X \leq \delta$ we have $b-X \leq b-a$, so $(b-X)^n \leq (b-a)^n$, whereas on $X > \delta$ we have $b-X \leq b-\delta$, so $(b-X)^n \leq (b-\delta)^n$. Therefore, from \cref{eq:EbmXsplit} and monotonicity of expectation, we obtain
    \begin{equation*}
        \E\left[(b-X)^n\right] \leq
        (b-a)^n \Prob(X \leq \delta) + (b - \delta)^n,
    \end{equation*}
    which after rearranging is equivalent to \cref{cor:seqsharperbound}.

    Next, if $\Prob(X \geq a + \delta) = 1$, then $b-X \leq b - a - \delta$ a.s., and hence $(b-X)^n \leq (b - a - \delta)^n$ a.s. for all $n \geq 0$. Taking expectations yields $\E[(b-X)^n] \leq (b-a-\delta)^n$, and applying \cref{thm:representation} and rearranging yields \cref{eq:gapexpdecay}, which proves \cref{cor:gapexpdecay}.

    Finally, to prove \cref{cor:tiltedmean}, define a probability measure $\nu_n$ on $[a,b]$ by
    \begin{equation*}
        \nu_n(B) := \frac{\E[(b-X)^n \Ind{X \in B}]}{\E[(b-X)^n]}, \qquad B \in \cB_{[a,b]},
    \end{equation*}
    and observe that
    \begin{equation*}
        (b-a) - \frac{(n+1)\eps_{n+1}}{\eps_n} 
        = \int_{[a,b]} \big( (b-a) - (b-x) \big) \nu_n(\dif x)
        = \int_{[a,b]} (x-a) \nu_n(\dif x).
    \end{equation*}
    Now, fix $\eta \in (0, (b-a))$. Since $x-a \leq \eta$ for $x \in [a, a+\eta]$ and $x-a \leq b-a$ for $x \in [a + \eta,b]$, we have
    \begin{equation*}
        \int_{[a,b]} (x-a) \nu_n(\dif x)
        \leq \eta \cdot \nu_n([a, a+\eta]) + (b-a) \cdot \nu_n([a+\eta,b])
        \leq \eta + (b-a) \cdot \nu_n([a+\eta,b]),
    \end{equation*}
    and thus
    \begin{equation}
        (b-a) - \frac{(n+1)\eps_{n+1}}{\eps_n}
        \leq 
        \eta + (b-a) \cdot \nu_n([a+\eta,b]).
    \end{equation}
    We claim that $\nu_n([a+\eta, b]) \to 0$ as $n \to \infty$. If this claim holds, then
    \begin{equation*}
        \limsup_{n \to \infty} \left( (b-a) - \frac{(n+1) \eps_{n+1}}{\eps_n}\right) \leq \eta,
    \end{equation*}
    and letting $\eta \to 0$, we therefore have 
    \begin{equation*}
        (b-a) - \frac{(n+1) \eps_{n+1}}{\eps_n} \longrightarrow 0,
    \end{equation*}
    which is equivalent to \cref{eq:ratiolimit}. To prove the claim, note that on the event $X \in [a+\eta, b]$ we have $b-X \leq b-a-\eta$, and so
    \begin{equation*}
        \E[(b-X)^n \Ind{X \in [a+\eta,b]}] 
        \leq (b-a-\eta)^n \, \Prob(X \in [a+\eta, b]) 
        \leq 
        (b-a-\eta)^n.
    \end{equation*}
    On the other hand, on the event $X \in [a, a+ \eta/2]$ we have $b-X \geq b - a - \eta/2$, so
    \begin{equation*}
        \E[(b-X)^n]
        \geq \E\left[(b-X)^n\Ind{X \in \left[a, a+ \frac{\eta}{2}\right]}\right]
        \geq \left(b-a - \frac{\eta}{2}\right)^n c_\eta,
    \end{equation*}
    where $c_\eta := \Prob(X \in [a, a + \eta/2]) \geq \Prob(X = a) > 0$. Therefore 
    \begin{equation*}
        \nu_n([a+\eta, b]) = \frac{\E[(b-X)^n \Ind{X \in [a+\eta, b]}]}{\E[(b-X)^n]}
        \leq \frac{(b-a-\eta)^n}{(b-a-\eta/2)^n c_\eta}
        =
        \frac{1}{c_\eta}\left( \frac{b-a-\eta}{b-a-\eta/2}\right)^n,
    \end{equation*}
    and since the last quantity in parentheses above is strictly less than $1$, we must have $\nu_n([a+\eta, b]) \to 0$ as $n \to \infty$, which concludes the proof of \cref{cor:tiltedmean}.
\end{proof}

The most important consequence of \cref{coro:normalizedsequence} is the convergence of the normalized sequence $\{\gamma_n\}_n$ to $\Prob(X = a)$. We now build on this result by demonstrating a polynomial rate of convergence under a mild regularity condition on the lower tail of $\mu$.

\begin{prop}\label{prop:gamma_rate}
    Assume there exist constants $\kappa>0$ and $C>0$ such that, for all $u\in[0,1]$,
    \begin{equation}\label{eq:lefttail-holder}
        \mu\bigl((a, a + (b-a)u] \bigr) \leq C u^{\kappa}.
    \end{equation}
    Then
    \begin{equation}\label{eq:holderconsequence}
        0 \leq \gamma_n- \Prob(X=a) 
        \leq C \Gamma(\kappa+1) \frac{n}{(n-1)^{\kappa+1}}
        \le C \Gamma(\kappa+1)n^{-\kappa}, \qquad n \geq 2.
    \end{equation}
\end{prop}

\begin{proof}
    Let $U := (X-a)/(b-a) \in [0,1]$ a.s. and observe that $\gamma_n=\E[(1-U)^n]$. Also let $p_0 := \Prob(U=0)=\Prob(X=a)$ and let $F_U(u) := \Prob(U \leq u)$ be the cdf of $U$. Integration by parts for Riemann-Stieltjes integrals \citep[][Theorem 7.6]{apostolMathematicalAnalysisModern1974} gives
    \begin{align*}
        \E[(1-U)^n]
        &=\int_{[0,1]} (1-u)^n \dif F_U(u)
        = (1-u)^n F_U(u)\Big|_{0}^{1}+ n\int_0^1 (1-u)^{n-1}F_U(u)\dif u\\
        &= p_0 + n\int_0^1 (1-u)^{n-1}\left(F_U(u)-p_0\right)\dif u,
    \end{align*}
    and hence
    \begin{equation}\label{eq:gamman0int}
        0 \leq \gamma_n - p_0 
        = n\int_0^1 (1-u)^{n-1}\left(F_U(u)-p_0\right)\dif u,
    \end{equation}
    where the lower bound is due to \cref{cor:lefttailID}.
    For any $u \in [0,1]$ we have
    \[
        F_U(u) - p_0 
        = \Prob(0 < U \leq u) 
        = \Prob(a < X \leq a + (b-a)u)
        = \mu\bigl((a, a + (b-a)u] \bigr) \leq C u^{\kappa}
    \]
    by \cref{eq:lefttail-holder}, and inserting this inequality into the integrand of \cref{eq:gamman0int} yields
    \begin{equation*}
        0 \leq \gamma_n - p_0 
        \leq Cn\int_0^1 (1-u)^{n-1}u^{\kappa}\dif u.
    \end{equation*}
    Applying the basic inequality $(1-u)^{n-1} \leq e^{-(n-1)u}$ for $u \in [0,1]$, we obtain
    \begin{equation*}
        0 \leq \gamma_n-p_0
        \leq Cn\int_0^1 e^{-(n-1)u}u^{\kappa}\dif u
        \leq Cn\int_0^{\infty} e^{-(n-1)u}u^{\kappa}\dif u
        = C\Gamma(\kappa+1)\frac{n}{(n-1)^{\kappa+1}},
    \end{equation*}
    which is the first upper bound in \cref{eq:holderconsequence}. The inequality $n/(n-1)^{\kappa+1}\le n^{-\kappa}$ for $n\ge 2$ gives the second inequality.
\end{proof}

\begin{rem}
    The condition in \cref{eq:lefttail-holder} is equivalent to
    \[
        \left|F_X(a + t) - F(a)\right| \leq \frac{C}{(b-a)^\kappa} t^\kappa
    \]
    for all $t = (b-a)u \in [0, b-a]$, which describes $\kappa$-H\"older continuity of $F_X$ at $a$ from the right.
\end{rem}

To conclude this section, we turn to generating functions. Ordinary generating functions of the raw moment sequence usually only have a finite radius of convergence; for example, when $X \sim \stdunif$, this generating function has the form 
\[
    G_X(z) = \sum_{n=0}^\infty \frac{z^n}{n+1} = -\frac{\log(1-z)}{z},
\]
which has a radius of convergence about $0$ of $1$. In contrast, \cref{thm:representation} implies that the ordinary generating function of the primitive sequence is always entire, which reflects the presence of the factor $1/n!$ in \cref{eq:basicrep}.

\begin{prop}\label{prop:EGF}
    Let
    \[
    G_X(z) := \sum_{n=0}^\infty \E[X^n] z^n
    \qquad\text{and}\qquad
    \tilde{G}_X(z) := \sum_{n=0}^\infty \eps_n z^n.
    \]
    Then the following hold:
    \begin{propitems}
        \item\label{prop:GradofC} $G_X$ has radius of convergence $1/\|X\|_\infty$, and in particular it is finite if and only if $\|X\|_\infty > 0$.
        \item\label{prop:analyticGentire} $\tilde{G}_X$ is an entire function.
    \end{propitems}
\end{prop}

\begin{proof}
    According to the Cauchy--Hadamard theorem, the radius of convergence of $G_X$ is
    \[
    \frac{1}{\limsup_{n\to\infty} |\E[X^n]|^{1/n}},
    \]
    so it is enough to show that
    \[
    \limsup_{n\to\infty} |\E[X^n]|^{1/n} = \|X\|_\infty.
    \]
    First, since $|X|\le \|X\|_\infty$ a.s., we have
    \[
    |\E[X^n]| \le \E[|X|^n] \le \|X\|_\infty^n,
    \]
    and therefore
    \begin{equation}\label{eq:limsup1}
        \limsup_{n\to\infty} |\E[X^n]|^{1/n} \le \|X\|_\infty.
    \end{equation}
    On the other hand, it is a standard fact of measure theory that $\|X\|_p \nearrow \|X\|_\infty$ as $p\to\infty$. In particular, along the even integers $p=2k$,
    \[
    \left|\E[X^{2k}]\right|^{1/(2k)}
    =
    \left(\E[|X|^{2k}]\right)^{1/(2k)}
    =
    \|X\|_{2k}
    \longrightarrow
    \|X\|_\infty,
    \]
    so
    \begin{equation}\label{eq:limsup2}
        \limsup_{n\to\infty} |\E[X^n]|^{1/n}
        \ge
        \limsup_{k\to\infty} |\E[X^{2k}]|^{1/(2k)}
        =
        \|X\|_\infty.
    \end{equation}
    Combining \cref{eq:limsup1,eq:limsup2} proves \cref{prop:GradofC}.
    
    For \cref{prop:analyticGentire}, fix $z\in\C$. By \cref{thm:representation},
    \[
    \tilde{G}_X(z)
    =
    \sum_{n=0}^\infty \frac{\E[(b-X)^n]}{n!} z^n.
    \]
    Moreover, since $0\le b-X\le b-a$ a.s.,
    \[
    \left|\frac{(b-X)^n z^n}{n!}\right|
    \le
    \frac{(b-a)^n |z|^n}{n!}.
    \]
    Hence for every $R>0$ and every $|z|\le R$,
    \[
    \sum_{n=0}^\infty \left|\frac{(b-X)^n z^n}{n!}\right|
    \le
    \sum_{n=0}^\infty \frac{(b-a)^n R^n}{n!}
    =
    e^{(b-a)R}.
    \]
    By the Weierstrass M-test, the series
    \[
    \sum_{n=0}^\infty \frac{(b-X)^n z^n}{n!}
    \]
    converges uniformly for $|z|\le R$, and therefore locally uniformly on $\C$. Since each partial sum is entire, the limit
    \[
    z \mapsto \sum_{n=0}^\infty \frac{(b-X)^n z^n}{n!}
    \]
    is entire. Taking expectations term-by-term, which is justified by the dominated convergence theorem on each closed disk, yields
    \[
    \tilde{G}_X(z)
    =
    \E\!\left[e^{(b-X)z}\right].
    \]
    Thus $\tilde{G}_X$ is entire.
\end{proof}

According to \cref{cor:sequencebounds}, we have
\[
\eps_n \le \frac{(b-a)^n}{n!}, \qquad n\ge 0.
\]
In particular, the coefficients of the primitive sequence decay at least factorially, uniformly over $\mu\in\cPab$. This same fact appears as a bound on the growth rate of $\tilde{G}_X$. To see this, observe that
\[
\sup_{|z|=r} |\tilde{G}_X(z)|
=
\sup_{|z|=r} \left|\E\!\left[e^{(b-X)z} \right]\right|
\leq
\sup_{|z|=r} \left|\E\!\left[e^{(b-X)\text{Re}(z)}\right]\right|
\le
e^{(b-a)r},
\]
and Cauchy's estimate therefore gives
\[
\eps_n
=
\frac{\tilde{G}_X^{(n)}(0)}{n!}
\le
\frac{e^{(b-a)r}}{r^n},
\qquad n\ge 1.
\]
Optimizing over $r>0$ yields
\[
\eps_n \le \left(\frac{e(b-a)}{n}\right)^n, \qquad n\ge 1.
\]
While this bound is slightly weaker than the bound from \cref{cor:sequencebounds}, it expresses the same factorial decay, since
\[
\left(\frac{e(b-a)}{n}\right)^n \sim \sqrt{2\pi n} \frac{(b-a)^n}{n!} \quad \text{ as } n \to \infty
\]
by Stirling's formula.

\section{Examples of primitive sequences}\label{sec:examples}

In this section, we compute the primitive sequence $\epsseq$ for several important classes of distributions. Our coverage here is not comprehensive; rather, we aim to illustrate how certain qualitative features of the distribution of $X$ (such as discreteness or flatness near the boundary points $a$ and $b$) are manifested in the structure --- and in particular, the rate of decay --- of its associated primitive sequence. We abuse notation slightly and reuse $\eps_n$ across the following examples, where the underlying distribution will be clear from context.

\begin{ex}\label{ex:beta}
    We begin with the Beta family. If $X \sim \mathrm{Beta}(\alpha, \beta)$, it is easy to see that $\E[(1-X)^n]  = \mathrm{B}(\alpha, \beta + n) / \mathrm{B}(\alpha, \beta)$,
    and therefore
    \begin{equation}\label{eq:sequencebeta}
        \eps_n = \frac{1}{n!} \frac{\mathrm{B}(\alpha, \beta + n)}{\mathrm{B}(\alpha, \beta)}.
    \end{equation}
    Using the asymptotic approximation $\mathrm{B}(x,y) \sim \Gamma(x) y^{-x}$ as $y \to \infty$ which follows from Formula~6.1.46 of \cite{abramowitz1988handbook}, \cref{eq:sequencebeta} implies that
    \[
        \eps_n \sim \frac{n^{-\alpha}}{n!}\frac{\Gamma(\alpha)}{\mathrm{B}(\alpha, \beta)}
    \]
    as $n \to \infty$, revealing an interesting asymmetry in the parameters of the distribution: the polynomial correction factor $n^{-\alpha}$ that modifies the baseline $1/n!$ rate of decay of $\eps_n$ (see \cref{cor:sequencebounds}) is governed only by the parameter $\alpha$, which controls the density of the $\mathrm{Beta}(\alpha, \beta)$ distribution at the lower endpoint $a = 0$. On the other hand, the parameter $\beta$, which controls the density's behaviour at $b = 1$, does not appear in the leading-order asymptotics. Thus, the rate at which the normalized sequence $\gamma_n = n! \eps_n$ (see \cref{coro:normalizedsequence}) decays to $0$ quantifies the ``thinness'' of the lower tail of the distribution: $\alpha < 1$ yields a slower decay (and a singularity of the density at $0$), while $\alpha > 1$ yields a faster decay. When $\alpha = 1$, we have $\eps_n = O(1/(n + 1)!)$ in the present $[0,1]$ setting; more generally, on an arbitrary compact interval $[a,b]$, the corresponding bound acquires the baseline factor $(b-a)^{n+1}$.
\end{ex}

\begin{ex}\label{ex:atomic}
    At the opposite extreme, consider a purely atomic  distribution supported on $k \geq 2$ points
    \[
    a \leq c_1 < c_2 < \cdots < c_k \leq b.
    \]
    If $\Prob(X = c_i) = \pi_i$ for $i=1,\ldots,k$, where $\pi_i \in [0,1]$ and $\sum_{i=1}^k \pi_i = 1$, then linearity of expectation gives
    \begin{equation}\label{eq:epspurelyatomicmixture}
        \eps_n = \sum_{i=1}^k \pi_i \frac{(b-c_i)^n}{n!}.
    \end{equation}
    Thus the primitive sequence is a convex combination of the primitive sequences generated by point masses. Because the bases $b-c_1,\ldots,b-c_k$ are distinct, the asymptotic behaviour of the sum is dominated by the largest base, namely $b-c_1$. In particular,
    \[
    \frac{(n+1)\eps_{n+1}}{\eps_n} \longrightarrow b-c_1
    \qquad\text{as } n\to\infty,
    \]
    in agreement with \cref{cor:lefttailID}. Moreover, if $\Delta := c_2-c_1 > 0$, then the contribution of the second atom relative to the first is governed by
    \[
    \left(\frac{b-c_2}{b-c_1}\right)^n
    =
    \left(1-\frac{\Delta}{b-c_1}\right)^n,
    \]
    so a larger separation between the two smallest atoms leads to faster asymptotic dominance by the leftmost atom.
    
    For countably-supported laws, a similar structure persists:
    \[
    \eps_n = \frac{1}{n!}\sum_{i=1}^\infty \pi_i (b-c_i)^n,
    \]
    and the asymptotic behaviour of $n!\eps_n$ is governed by atoms near the infimum of the support of $X$. If the infimum is attained, then the corresponding atom dominates; otherwise, atoms accumulating near the infimum determine the decay of $n!\eps_n$.
    
    We also note that the normalized sequence $\{\gamma_n\}$ from \cref{cor:lefttailID} admits a simple generating function in the atomic case, and particularly so when $b-a = 1$. In this case, writing $q_i := b-c_i$, we have
    \[
    \sum_{n=0}^\infty \gamma_n z^n
    =
    \sum_{n=0}^\infty n!\eps_n z^n
    =
    \sum_{i=1}^\infty \frac{\pi_i}{1-q_i z},
    \qquad
    |z| < \frac{1}{\sup_i q_i}.
    \]
\end{ex}

\begin{ex}\label{ex:piecewisepolynomial}
    Primitive sequences are also relatively easy to compute for piecewise polynomial (i.e., spline) densities, since the antiderivative of a polynomial is another polynomial. For example, if $X$ has a triangular density on $[0,1]$ with a mode at $0$ or $1$, then $1-X$ has a piecewise quadratic cdf, and $\E[(1-X)^n]$ can be computed in closed form as a rational function of $n$. The resulting primitive sequence again decays like $1/n!$, but with polynomial corrections that reflect the non-smoothness of the density at its mode. More generally, if the knots of the density are $a = \xi_0 < \xi_1 < \cdots < \xi_m = b$ and on each subinterval $[\xi_{i-1}, \xi_i]$ the density of $X$ is given by a polynomial $p_i$, then repeated integration by parts shows that $\eps_n$ admits an expansion of the form
    \begin{equation}\label{eq:piecewisepoly}
        \eps_n = \sum_{j=0}^{d_m} \frac{1}{(n+j+1)!}\left[p_1\anti{j}(a) \cdot (b-a)^{n+j+1} + \sum_{i=1}^{m-1} \left(p_{i+1}\anti{j}(\xi_i) - p_i\anti{j}(\xi_i) \right) \cdot (b-\xi_i)^{n+j+1}\right],
    \end{equation}
    where $d_m := \max\{\mathrm{deg}(p_1), \ldots, \mathrm{deg}(p_m)\}$. 
    In particular, if $m=1$ and $[a,b]=[0,1]$ with
    \[
    p_1(x)=\sum_{j=0}^d a_j x^j,
    \]
    then
    \[
    \eps_n = \sum_{j=0}^d \frac{a_j\,j!}{(n+j+1)!}.
    \]
    It is interesting to compare \cref{eq:epspurelyatomicmixture} to \cref{eq:piecewisepoly}. For purely atomic distributions, each atom $c_i$ contributes to $\eps_n$ in the form of the exponential term $(b-c_i)^n/n!$. In contrast, for piecewise polynomial densities, a jump in the $j$th derivative of the density contributes to $\eps_n$ at the $1/(n+j+1)!$ scale, so each knot $\xi_i$ is effectively a ``smoothed'' atom. In this sense, the primitive sequence detects smoothness: if the density is $k$-times continuously differentiable at a knot $\xi_i$, then $p_{i+1}\anti{j}(\xi_i) - p_i\anti{j}(\xi_i) = 0$ for all $j \leq k$, and hence the contribution of $\xi_i$ to $\eps_n$ is pushed into higher-order terms starting at $1/(n+k+2)!$.
\end{ex}

\cref{ex:beta,ex:atomic,ex:piecewisepolynomial} demonstrate that smooth densities lead to primitive sequences which decay rapidly, while mass near the boundary slows down the decay in a quantifiable fashion; discrete components, when present, contribute exponential terms of the form $(b-c)^n/n!$. Another useful feature of the primitive sequence is its invariance with respect to mixtures. If $\mu = \int_{A} \mu_{\alpha} \nu(\dif \alpha)$ where $\nu$ is a probability measure on $A \subseteq \R$ and $\mu_\alpha \in \cPab$ for each $\alpha \in A$, then Fubini's theorem implies
\begin{equation}\label{eq:closedundermixtures}
    \eps_n(\mu) = \int_A \eps_n(\mu_\alpha) \nu(\dif \alpha).
\end{equation}
Thus, the set of primitive sequences is convex, and through \cref{thm:representation} this convexity is exactly the classical convexity of (scaled) Hausdorff moment sequences.

\section{Admissible sequences, the primitive sequence map, and identifiability}\label{sec:admissiblesequences}

In \cref{sec:reptheorem,sec:examples} we started from a probability measure $\mu \in \cPab$ and constructed its associated primitive sequence $\epsseq$. We now reverse this perspective and study the image of $\cPab$ under this construction. The main result of this section is that, through the identification in \cref{thm:representation}, admissibility and topology of primitive sequences are governed by the classical Hausdorff moment problem.

We call a real-valued sequence $\epsseq$ \emph{admissible} for the interval $[a,b]$ if there exists a cdf $F$ with $F(a-) = 0$ and $F(b) = 1$ such that $\eps_n = F\anti{-n}(b)$ for each $n \geq 0$. Equivalently, by \cref{thm:representation}, $\epsseq$ is admissible for $[a,b]$ if and only if there exists a probability measure $\mu \in \cPab$ such that
\begin{equation*}
    \eps_n = \frac{1}{n!} \int_{[a,b]} (b-x)^n \mu(\dif x), \qquad n \geq 0.
\end{equation*}
Thus admissibility is a global constraint on the entire sequence $\epsseq$; in particular, the elementary properties in \cref{cor:sequencebounds,cor:log-convex} are necessary but not sufficient. Our goal is therefore to characterize admissibility without referring explicitly to a representing measure. This characterization is provided by the Hausdorff moment problem \citep{hausdorff1921summationsmethoden}; see also Chapter VII, Section 3 of \cite{fellerIntroductionProbabilityTheory1968} and \cite{schmudgenMomentProblem2017}. The next theorem essentially translates that classical criterion into primitive sequence coordinates.

\begin{thm}\label{thm:admissibilitycond}
    A sequence $\epsseq$ is admissible for $[a,b]$ if and only if the normalized sequence $\{\gamma_n\}_n$ given by
    \begin{equation*}
        \gamma_n = \frac{n! \eps_n}{(b-a)^n}, \qquad n \geq 0
    \end{equation*}
    has $\gamma_0 = 1$ and is completely monotone --- that is,
    \[
    (-1)^k (\Delta^k \gamma)_n \geq 0, \qquad n,k \geq 0,
    \]
where the forward difference operator $\Delta$ is defined by $(\Delta \gamma)_n = \gamma_{n+1} - \gamma_n$ and $\Delta^k$ denotes its $k$-fold iterate.
\end{thm}
\begin{proof}
    Suppose that $\epsseq$ is admissible for $[a,b]$. Then there exists a random variable $X$ with $X \in [a,b]$ a.s. such that $\eps_n = \E[(b-X)^n]/n!$ by \cref{thm:representation}. Therefore,
    \begin{equation*}
        \gamma_n = \frac{n! \eps_n}{(b-a)^n} = \frac{\E[(b-X)^n]}{(b-a)^n} = \E\left[\left(\frac{b-X}{b-a}\right)^n\right],
    \end{equation*}
    for every $n \geq 0$, so that $\{\gamma_n\}_n$ is the raw moment sequence of the random variable $(b-X)/(b-a) \in [0,1]$ a.s. and $\gamma_0 = 1$. By the classical Hausdorff moment problem, this condition is equivalent to complete monotonicity of $\{\gamma_n\}_n$.

    Conversely, suppose that $\{\gamma_n\}_n$ is completely monotone and $
    \gamma_0 = 1$. By the Hausdorff moment problem, there exists a probability measure $\nu$ on $[0,1]$ and a random variable $Y \sim \nu$ such that $\gamma_n = \E[Y^n]$ for $n \geq 0$. Define $X := b - (b-a)Y$ so that $X \in [a,b]$ a.s. and let $F_X$ be its cdf. Then by \cref{thm:representation},
    \begin{equation*}
        F_X\anti{-n}(b)
        = \frac{\E[(b-X)^n]}{n!} 
        = \frac{\E[((b-a)Y)^n]}{n!}
        =
        \frac{(b-a)^n  \gamma_n}{n!}
        =
        \eps_n.
    \end{equation*}
    Thus, the sequence $\epsseq$ arises from a cdf supported on $[a,b]$ (i.e., $F_X$), and is therefore admissible.
\end{proof}

We now turn to geometry by studying the map $\mu \mapsto \{\eps_n(\mu)\}_n$. To formalize ideas, let $\cAab$ be the set of admissible sequences for $[a,b]$, which by \cref{thm:representation} can be written as
\begin{equation*}
    \cAab = \{ \epsseq \in \R^\N: \text{there exists } X \in [a,b] \text{ a.s. with } \eps_n = \E[(b-X)^n]/n!\}.
\end{equation*}
Define the \emph{primitive sequence map}
\begin{equation*}
    S: \cPab \to \prod_{n =0}^\infty \left[0, \frac{(b-a)^n}{n!}\right], \qquad S(\mu) = \{\eps_n(\mu)\}_n,
\end{equation*}
where
    \begin{equation*}
        \eps_n(\mu) := \frac{1}{n!} \int_{[a,b]} (b-x)^n \mu(\dif x) 
        = \frac{\E[(b-X)^n]}{n!}, \qquad X \sim \mu,
    \end{equation*}
so that $\cAab = S(\cPab)$. The next lemma records two basic properties of $S$:

\begin{lemm}\label{lemm:Taffine}
    The primitive sequence map $S$ is affine on $\cPab$ and continuous when $\cPab$ is equipped with the weak topology and $\prod_{n = 0}^\infty \left[0, (b-a)^n/n!\right]$ is equipped with the product topology. In particular, if $\mu_k \wc \mu$, then $\eps_n(\mu_k) \to \eps_n(\mu)$ for every $n \geq 0$. 
\end{lemm}
\begin{proof}
    That $S$ is affine follows from linearity of the integral: if $\mu_1, \mu_2 \in \cPab$, then for $\lambda \in [0,1]$ and $n \geq 0$ we have
    \begin{align*}
        \eps_n( \lambda \mu_1 + (1-\lambda)\mu_2) &= \frac{1}{n!} \int_{[a,b]} (b-x)^n (\lambda \mu_1(\dif x) + (1-\lambda)\mu_2(\dif x)) \\
        &= \lambda \frac{1}{n!} \int_{[a,b]} (b-x)^n \mu_1(\dif x) + (1-\lambda)\frac{1}{n!} \int_{[a,b]} (b-x)^n \mu_2(\dif x) \\
        &= \lambda \eps_n(\mu_1) + (1-\lambda) \eps_n(\mu_2),
    \end{align*}
    and so $S$ is affine. To prove continuity, observe that $x \mapsto (b-x)^n/n!$ is bounded and continuous on $[a,b]$; therefore, if $\mu_k \wc \mu$, then the Portmanteau lemma \cite[][Lemma 2.2]{vandervaartAsymptoticStatistics1998} implies
    \begin{equation*}
        \frac{1}{n!}\int_{[a,b]} (b-x)^n \mu_k(\dif x) \longrightarrow \frac{1}{n!}\int_{[a,b]} (b-x)^n \mu(\dif x), \qquad n \geq 0,
    \end{equation*}
    as required.
\end{proof}

\cref{lem:Pabcompact} below records the compactness of $\cPab$, after which the corresponding geometric and topological properties of $\cAab$ follow immediately.

\begin{lemm}\label{lem:Pabcompact}
    $\cPab$ is compact with respect to the weak topology.
\end{lemm}
\begin{proof}
    $\cPab$ is tight because $[a,b]$ is compact; indeed, for any $\mu \in \cPab$, given any $\eta > 0$, taking the compact set $K = [a,b]$ gives $\mu(K) = 1 > 1 - \eta$. By Prokhorov's theorem, $\cPab$ is relatively compact. Moreover, if $\mu_k \in \cPab$ for $k \geq 1$ and $\mu_k \wc \mu^*$, then the Portmanteau lemma gives $\mu^*([a,b]) \geq \limsup_{k \to \infty} \mu_k([a,b]) = 1$ because $[a,b]$ is closed, and hence $\mu^*([a,b]) = 1$ and so $\mu^* \in \cPab$. Thus $\cPab$ is closed, and hence compact.
\end{proof}

\begin{coro}\label{cor:convexcompact}
    $\cAab$ is both convex and compact with respect to the product topology on the product space $\prod_{n \geq 0} [0, (b-a)^n/n!]$.
\end{coro}
\begin{proof}
    Let $\{\eps\anti{1}_n\}_n, \{\eps\anti{2}_n\}_n \in \cAab$ so that there exist $\mu_1, \mu_2 \in \cPab$ with $S(\mu_i) = \{\eps\anti{i}_n\}_n$ for $i=1,2$. Since $\cPab$ is convex, we have $\lambda \mu_1 + (1-\lambda)\mu_2 \in \cPab$ for any $\lambda \in [0,1]$, and since $S$ is affine by \cref{lemm:Taffine},
    \[
        \lambda \{\eps\anti{1}_n\}_n + (1-\lambda)\{\eps\anti{2}_n\}_n
        =
        S(\lambda \mu_1 + (1-\lambda)\mu_2)
        \in \cAab.
    \]
    Hence $\cAab$ is convex. Since $\cPab$ is compact with respect to the weak topology by \cref{lem:Pabcompact}, and $S$ is continuous by \cref{lemm:Taffine}, its image
    \[
        \cAab = S(\cPab)
    \]
    is compact with respect to the product topology.
\end{proof}

The compactness of $\cAab$ guarantees that continuous functionals on $\cAab$ attain their extrema, and the convexity of $\cAab$ at least suggests that these extrema should occur at the extreme points of $\cAab$. This naturally raises the question of whether such extrema are unique, or more generally, whether an underlying measure $\mu \in \cPab$ is identifiable from the primitive sequence $\{\eps_n(\mu)\}_n$. \cref{thm:admissibilitycond} by itself does not preclude different measures from sharing the same primitive sequence. However, \cref{lemm:Tinjective} below shows that this cannot happen:

\begin{lemm}\label{lemm:Tinjective}
    The primitive sequence map $S$ is injective: if $\mu_1, \mu_2 \in \cPab$ are such that $S(\mu_1) = S(\mu_2)$, then $\mu_1 = \mu_2$; thus, the sequence $\{\eps_n(\mu)\}_n$ uniquely determines $\mu$.
\end{lemm}
\begin{proof}
    If $S(\mu_1) = S(\mu_2)$, then 
    \begin{equation*}
        \frac{1}{n!}\int_{[a,b]} (b-t)^n \mu_1(\dif t) = \frac{1}{n!}\int_{[a,b]} (b-t)^n \mu_2(\dif t), \qquad n \geq 0.
    \end{equation*}
    Through finitely many affine changes of variable and linearity, it follows that 
    \begin{equation*}
        \int_{[a,b]} p(t) \mu_1(\dif t) = \int_{[a,b]} p(t) \mu_2(\dif t)
    \end{equation*}
    for any polynomial $p$, so by the Weierstrass approximation theorem, we have
    \begin{equation*}
        \int_{[a,b]} f(t) \mu_1(\dif t) = \int_{[a,b]} f(t) \mu_2(\dif t)
    \end{equation*}
    for any $f \in C([a,b])$. By the Riesz representation theorem \citep[][Theorem 2.14]{rudinRealComplexAnalysis2013}, positive Borel measures on $[a,b]$ are uniquely determined by their integrals against continuous test functions; it follows that $\mu_1 = \mu_2$.
\end{proof}

We can immediately strengthen the identifiability shown by \cref{lemm:Tinjective} to limiting cases by showing that termwise convergence of primitive sequences is equivalent to weak convergence of measures, formalizing how primitive sequences parametrize $\cPab$. We first need the following simple consequence of Prokhorov's theorem \citep[][Theorem 5.1]{billingsleyConvergenceProbabilityMeasures1999}, which will be used again in \cref{sec:finitetruncation}:

\begin{thm}\label{thm:Thomeo}
    Equip $\cPab$ with the weak topology and $\cAab$ with the subspace topology inherited from the product topology on $\prod_{n = 0}^\infty [0, (b-a)^n/n!]$. Then the primitive sequence map $S$ is a homeomorphism between $\cPab$ and $\cAab$; in particular, for $\mu_k, \mu \in \cPab$ with $k \geq 1$, 
    \begin{equation*}
        \mu_k \wc \mu \quad \Longleftrightarrow \quad \eps_n(\mu_k) \to \eps_n(\mu), \quad n \geq 0.
    \end{equation*}
\end{thm}
\begin{proof}
    First, note that since $S$ is injective by \cref{lemm:Tinjective} and surjective by definition, it is a  bijection from $\cPab$ to $\cAab$. Moreover, by \cref{lemm:Taffine}, $S$ is continuous when $\prod_{n=0}^\infty [0, (b-a)^n/n!]$ is equipped with the product topology, and this continuity is preserved when we restrict $S$ to the subspace $\cAab$; thus, $S$ is a continuous bijection from $\cPab$ onto $\cAab$. Now, observe that each interval $[0, (b-a)^n/n!]$ is Hausdorff, and so the product space $\prod_{n=0}^\infty [0, (b-a)^n/n!]$ is also Hausdorff \citep[Theorem 31.2(a)]{munkresTopology2018};  therefore its subspace $\cAab$ is Hausdorff as well. Since $\cPab$ is compact by \cref{lem:Pabcompact}, $S$ is a continuous bijection from a compact space to a Hausdorff space. This in turn implies that $S$ is a homeomorphism onto its image $\cAab$ \citep[Theorem 26.2]{munkresTopology2018}.

    Next, if $\mu_k \wc \mu$, then $\eps_n(\mu_k) \to \eps_n(\mu)$ as $k \to \infty$ for each $n \geq 0$ by \cref{lemm:Taffine}; that is, $S(\mu_k) \to S(\mu)$ in the product topology. Conversely, if $\eps_n(\mu_k) \to \eps_n(\mu)$ for each $n \geq 0$, then $S(\mu_k) \to S(\mu)$ in the product topology, and by continuity of the inverse mapping $S^{-1}$, we obtain 
    \begin{equation}
        \mu_k = S^{-1}(S(\mu_k)) \wc S^{-1}(S(\mu)) = \mu,
    \end{equation}
    which completes the proof.
\end{proof}

The results of \cref{lemm:Tinjective,thm:Thomeo} can be combined neatly:

\begin{coro}
    An admissible sequence $\epsseq \in \cAab$ uniquely determines a probability measure $\mu \in \cPab$.
\end{coro}
\begin{proof}
    Since the primitive sequence map $S$ is a homeomorphism from $\cPab$ to $\cAab$ by \cref{thm:Thomeo}, every $\epsseq \in \cAab$ has a unique preimage under $S$.
\end{proof}

Finally, we identify the extreme points of $\cAab$. Informally, these correspond to ``irreducible'' primitive sequences that cannot be expressed as (non-trivial) convex combinations of other admissible sequences. In fact, the only such primitive sequences are those generated by point masses:

\begin{prop}\label{prop:extremepoints}
    The extreme points of $\cAab$ are precisely the sequences arising from point masses. That is, $\epsseq \in \mathrm{ext}(\cAab)$ if and only if there exists $x \in [a,b]$ such that $\eps_n = (b-x)^n/n!$ for all $n \geq 0$.
\end{prop}
\begin{proof}
    First note that if $X = x \in [a,b]$ a.s., then $\eps_n = \E[(b-x)^n]/n! = (b-x)^n/n!$ by \cref{thm:representation}. We claim this sequence is extreme. Indeed, suppose it can be written as a convex combination of two sequences in $\cAab$; that is,
    \begin{equation*}
        \frac{(b-x)^n}{n!} = \lambda \eps\anti{1}_n + (1-\lambda) \eps\anti{2}_n, \qquad n \geq 0
    \end{equation*}
    for some $\{\eps\anti{1}\}_n, \{\eps\anti{2}\}_n \in \cAab$ and $\lambda \in (0,1)$. Since these sequences are admissible for $[a,b]$, there exist $\mu_1, \mu_2 \in \cPab$ such that $\eps\anti{i}_n = \eps_n(\mu_i)$ for $i = 1,2$. Because the primitive sequence map $S$ is affine by \cref{lemm:Taffine}, we have
    \begin{equation*}
        S(\delta_x) = \epsseq
        = \left\{ \lambda \eps\anti{1}_n + (1-\lambda) \eps\anti{2}_n \right\}_n
        = \lambda \{\eps\anti{1}_n\}_n + (1-\lambda) \{\eps\anti{2}_n\}_n
        = S(\lambda \mu_1 + (1-\lambda)\mu_2).
    \end{equation*}
    By \cref{lemm:Tinjective}, it follows that $\delta_x = \lambda \mu_1 + (1-\lambda)\mu_2$. Evaluating both sides at $\{x\}$ gives
    \begin{equation*}
        1 = \lambda \mu_1(\{x\}) + (1-\lambda) \mu_2(\{x\})
    \end{equation*}
    and hence $\mu_1(\{x\}) = \mu_2(\{x\}) = 1$, whence $\mu_1 = \mu_2 = \delta_x$. Therefore $S(\mu_1) = S(\mu_2) = S(\delta_x)$; that is, $\{\eps\anti{1}\}_n = \{\eps\anti{2}\}_n = \epsseq$, so $\epsseq \in \mathrm{ext}(\cAab)$.

    Conversely, suppose that $\epsseq \in \cAab$ is generated by a measure $\mu \in \cPab$ that is \emph{not} a point mass. Then there exists $B \in \cB_{[a,b]}$ with $0 < \mu(B) < 1$. Define the conditional laws $\mu_1 := \mu(\cdot \mid B)$ and $\mu_2 := \mu(\cdot \mid B^c)$ so that $\mu = \lambda_B \mu_1 + (1-\lambda_B) \mu_2$ with $\mu_1 \neq \mu_2$, where $\lambda_B \in (0,1)$. Since $S$ is affine by \cref{lemm:Taffine}, we have $S(\mu) = \lambda_B S(\mu_1) + (1-\lambda_B) S(\mu_2)$, and it follows that $S(\mu_1) \neq S(\mu_2)$ (otherwise, we would have $\mu_1 = \mu_2$ by \cref{lemm:Tinjective}, a contradiction). Hence $S(\mu)$ is a non-trivial convex combination of $S(\mu_1)$ and $S(\mu_2)$. That is, $\{\eps_n\}$ is a non-trivial convex combination of two distinct points $S(\mu_1), S(\mu_2) \in \cAab$, and therefore $\epsseq \not\in \mathrm{ext}(\cAab)$.    
\end{proof}

\section{Finite truncation and applications}\label{sec:finitetruncation}

In \cref{sec:admissiblesequences}, we showed that probability measures in $\cPab$ are uniquely determined by their full primitive sequences. However, any practical use of a primitive sequence would naturally involve only finitely many of its terms. This leads naturally to considering optimization problems over the class of probability measures whose first $m$ primitive terms are fixed --- a specialization of the classical truncated moment problem (see, for example, \cite{curtoRecursivenessPositivityTruncated1991, schmudgenMomentProblem2017, kreinMarkovMomentProblem1977}) to the coordinates defined by the primitive sequence introduced here. Our aim in this section is to derive the corresponding extremal principle with finitely supported measures in these coordinates, and to present two concrete applications.

It will be convenient to define the set-valued map $\cF_m:\R^m \to 2^{\cPab}$ by
\[
\cF_m(c_{1:m})
:=
\left\{
\mu \in \cPab :
\frac{1}{j!}\int_{[a,b]} (b-t)^j\,\mu(\dif t)=c_j
\text{ for } j=1,\ldots,m
\right\}.
\]
That is, $\cF_m(c_{1:m})$ is the set of probability measures in $\cPab$ whose first $m$ primitive coordinates equal $c_1,\ldots,c_m$.

\begin{thm}\label{thm:finitetruncation}
    Suppose that $c_1, \ldots, c_m$ are the first $m \geq 1$ terms of an admissible sequence for $[a,b]$, and let $L:\cPab \to \R$ be a functional that is affine on $\cFmeps$. If $L$ is also upper semicontinuous on $\cFmeps$, then 
    \[
        \sup_{\mu \in \cFmeps} L(\mu)
    \]
    is attained by a distribution supported on at most $m+1$ points in $[a,b]$, while if $L$ is lower semicontinuous on $\cFmeps$, then 
    \[
        \inf_{\mu \in \cFmeps} L(\mu)
    \]
    is attained by a distribution supported on at most $m+1$ points in $[a,b]$.
\end{thm}

\begin{rem}
    \cite{winklerExtremePointsMoment1988} develops general results on extreme points of convex sets of probability measures defined by finitely many integral constraints of the form $\int_{\cX} f_j(x) \mu(\dif x) = c_j$, where $f_j$ is real and measurable on a measurable space $\cX$. Under relatively mild conditions, the extreme points of these sets are finitely supported, and the support size is bounded in terms of the number of constraints \citep[][Theorem~2.1]{winklerExtremePointsMoment1988}. A complementary and classical result is Tchakaloff's theorem \citep[][Theorem~1.24]{tchakaloffFormulesCubatureMecanique1957,schmudgenMomentProblem2017}, which ensures the existence of a finitely supported measure matching a given positive linear functional on a finite-dimensional function space. Theorem~\ref{thm:finitetruncation} may be viewed as a specialization of these general results to the functions $f_j(x) = (b-x)^j/j!$, $j=1,\dots,m$, together with the observation that semicontinuous affine functions attain their extrema at extreme points under this set of constraints. However, in our setting, the following self-contained proof reduces the key step to the elementary argument underlying Carath\'eodory's theorem that any $m+2$ points in $\R^m$ are affinely dependent.
\end{rem}

\begin{proof}[Proof of \cref{thm:finitetruncation}.]
    We treat the case that $L$ is upper semicontinuous; the proof is analogous when $L$ is lower semicontinuous. To begin, equip $\cPab$ with the weak topology. $\cPab$ is compact by \cref{lem:Pabcompact}, and because for each $j \geq 0$ the function $t \mapsto (b-t)^j/j!$ is bounded and continuous on $[a,b]$, the map
    \begin{equation*}
        \mu \mapsto \frac{1}{j!}\int (b-t)^j \mu(\dif t)
    \end{equation*}
    is continuous. Therefore $\cFmeps$ is closed in $\cPab$ and hence compact, so $\sup_{\mu \in \cFmeps} L(\mu)$ is attained by at least one $\tilde{\mu} \in \cFmeps$. Let $K := \{\mu \in \cFmeps : L(\mu) = L(\tilde{\mu})\}$ and observe that $K$ is non-empty and convex. Moreover, because $L$ is upper semicontinuous, $K = L^{-1}([L(\tilde{\mu}),\infty))$ is a closed subset of the compact set $\cFmeps$ and hence compact itself. By the Krein--Milman theorem \citep[][Theorem 3.22]{rudinFunctionalAnalysis1973}, $K$ contains an extreme point $\mu^*$. Moreover, if $\mu^* = \lambda \mu_1 + (1-\lambda) \mu_2 \in K$ with $\lambda \in (0,1)$ and $\mu_1, \mu_2 \in \cFmeps$, then affineness of $L$ implies $L(\tilde{\mu}) = \lambda L(\mu_1) + (1-\lambda)L(\mu_2)$, whence $L(\mu_1) = L(\mu_2) = L(\tilde{\mu})$, and so $\mu_1, \mu_2 \in K$. Thus $K$ is a face of $\cFmeps$, and therefore $\mathrm{ext}(K) \subseteq \mathrm{ext}(\cFmeps)$. Hence $\mu^* \in \mathrm{ext}(\cFmeps)$, and it remains to show that $\mu^*$ is supported on at most $m+1$ points.

    To this end, suppose for a contradiction that there exist $m+2$ disjoint $A_1, \ldots, A_{m+2} \in \cB_{[a,b]}$ with $\mu^*(A_i) > 0$ for $i=1, \ldots, m+2$. For each such $i$, define the conditional law $\mu_i := \mu^*(\cdot \mid A_i) \in \cPab$, the weight $\lambda_i := \mu^*(A_i) \in (0,1]$, and the vector
    \begin{equation*}
        \bv_i := \left(\int (b-t) \mu_i(\dif t), \frac{1}{2!}\int (b-t)^2 \mu_i(\dif t), \ldots, \frac{1}{m!}\int (b-t)^m \mu_i(\dif t) \right) \in \R^m.
    \end{equation*}
    Since $\R^m$ is $m$-dimensional, the $m+2$ vectors $\bv_1,\ldots,\bv_{m+2}$ are affinely dependent; that is, there exist coefficients $\alpha_1,\ldots,\alpha_{m+2}\in\R$, not all zero, such that
    \begin{equation}\label{eq:affinelyindependent}
        \sum_{i=1}^{m+2} \alpha_i = 0
        \qquad\text{and}\qquad
        \sum_{i=1}^{m+2} \alpha_i v_i = 0.
    \end{equation}
    Choose $\delta>0$ small enough that
    \[
    \min\{\lambda_i+\delta\alpha_i,\lambda_i-\delta\alpha_i\}\ge 0
    \qquad\text{for each } i,
    \]
    and define
    \[
    \mu^+
    :=
    \mu^*\big|_{[a,b]\setminus A}
    +
    \sum_{i=1}^{m+2}(\lambda_i+\delta\alpha_i)\mu_i,
    \qquad
    \mu^-
    :=
    \mu^*\big|_{[a,b]\setminus A}
    +
    \sum_{i=1}^{m+2}(\lambda_i-\delta\alpha_i)\mu_i,
    \]
    where $A=\bigcup_{i=1}^{m+2} A_i$. We have that
    \[
        \mu^\pm([a,b]) = \mu^*([a,b] \setminus A) + \sum_{i=1}^{m+2} (\lambda_i \pm \delta \alpha_i) = \left(1 - \sum_{i=1}^{m+2} \lambda_i \right) + \sum_{i=1}^{m+2} \lambda_i \pm \delta \sum_{i=1}^{m+2} \alpha_i = 1
    \]
    due to the first constraint in \cref{eq:affinelyindependent}, so $\mu^+,\mu^- \in \cPab$ with $\mu^+ \neq \mu^-$. Moreover, for any $B \in \cB_{[a,b]}$, 
    \begin{align*}
        \frac{\mu^+(B) + \mu^-(B)}{2}= \mu^*(B \setminus A) + \sum_{i=1}^{m+2} \lambda_i \mu_i(B)
        = \mu^*(B \setminus A) + \mu^*(A \cap B)
        = \mu^*(B),
    \end{align*}
    so $\mu^* = (\mu^+ + \mu^-)/2$. Now, for each $j = 1, \ldots, m$, we have
    \begin{align*}
        \int (b-t)^j \mu^\pm(\dif t) &=  \int_{[a,b] \setminus A} (b-t)^j \mu^*(\dif t) + \sum_{i=1}^{m+2} (\lambda_i \pm \delta \alpha_i) \int (b-t)^j \mu_i(\dif t)\\
        &= \int_{[a,b] \setminus A} (b-t)^j \mu^*(\dif t) +\sum_{i=1}^{m+2} \int_{A_i} (b-t)^j \mu^*(\dif t) \pm  \delta \sum_{i=1}^{m+2} \alpha_i\int (b-t)^j \mu_i(\dif t)\\
        &= \int_{[a,b] \setminus A} (b-t)^j \mu^*(\dif t) +\sum_{i=1}^{m+2} \int_{A_i} (b-t)^j \mu^*(\dif t)\\
        &= \int (b-t)^j \mu^*(\dif t),
    \end{align*}
    where the third equality holds due to the second constraint in \cref{eq:affinelyindependent}. Thus
    \begin{equation*}
        \frac{1}{j!}\int (b-t)^j \mu^\pm(\dif t) = \frac{1}{j!}\int (b-t)^j \mu^*(\dif t),
    \end{equation*}
    and it follows that $\mu^+, \mu^- \in \cFmeps$. Therefore, $\mu^*$ is a non-trivial convex combination of two different measures in $\cFmeps$, which contradicts the fact that $\mu^* \in \mathrm{ext}(\cFmeps)$. Thus, $\mu^*$ is supported on at most $m+1$ points in $[a,b]$.
\end{proof}

\cref{thm:finitetruncation} expresses a familiar phenomenon from truncated moment problems in primitive coordinates. Under finitely many primitive constraints, extremizers of affine semicontinuous functionals may be taken to be finitely supported, with support size at most $m+1$. We now apply \cref{thm:finitetruncation} in two related settings.

\begin{coro}\label{cor:momentmatching}
    Suppose that $c_1, \ldots, c_m$ are the first $m \geq 1$ terms of an admissible sequence for $[a,b]$. Then for every integer $k \geq 0$, 
    \begin{equation*}
        M_m^+(k; c_{1:m}) := \sup_{\mu \in \cFmeps} \int_{[a,b]} \frac{(b-t)^k}{k!} \mu(\dif t)
        \end{equation*}
    and
    \begin{equation*}
        M_m^-(k; c_{1:m}) := \inf_{\mu \in \cFmeps} \int_{[a,b]} \frac{(b-t)^k}{k!}  \mu(\dif t) 
    \end{equation*}
    are attained by $(m+1)$-atomic measures. Equivalently, if $X$ is any random variable on $[a,b]$ with given moments $\E[X^k]$ for $k=1,\ldots,m$, then sharp bounds on $\E[p(X)]$ for any degree-$r$ polynomial $p$ are attained by discrete random variables supported on at most $m+1$ points.
\end{coro}

\begin{rem}
    For $k \in \{0,\ldots,m\}$, the quantity $\int_{[a,b]} (b-t)^k/k! \mu(\dif t)$ is fixed by definition of $\cFmeps$, so in this case $M_m^+(k; c_{1:m}) = M_m^-(k; c_{1:m}) = c_k$.
\end{rem}

\begin{proof}[Proof of \cref{cor:momentmatching}.]
    Define $M_k:\cPab \to \R$ by 
    \[
        M_k(\mu) := \int_{[a,b]} \frac{(b-t)^k}{k!}  \mu(\dif t).
    \]
    Because $t \mapsto (b-t)^k/k!$ is bounded and continuous on $[a,b]$, the functional $M_k$ is continuous by the Portmanteau lemma. In particular, $M_k$ is both upper and lower semicontinuous, so by \cref{thm:finitetruncation}, $M_m^+(k; c_{1:m}) = \sup_{\mu \in \cFmeps} M_k(\mu)$ and $M_m^-(k; c_{1:m}) = \inf_{\mu \in \cFmeps} M_k(\mu)$ are each attained by 
    measures in $\cPab$ supported on at most $m+1$ points.
\end{proof}

Our next application involves bounding cdfs:

\begin{coro}\label{cor:cdfbounding}
    Suppose that $c_1, \ldots, c_m$ are the first $m \geq 1$ terms of an admissible sequence for $[a,b]$. Then for every $x_0 \in [a,b]$,
    \begin{equation*}
    \bar{F}_m(x_0; c_{1:m}) := \sup_{\mu \in \cFmeps} \mu([a,x_0]) 
    \quad \text{and} \quad 
    \underline{F}_m(x_0; c_{1:m}) := \inf_{\mu \in \cFmeps} \mu([a,x_0))
\end{equation*}
are attained by $(m+1)$-atomic measures. Equivalently, if $X$ is any random variable on $[a,b]$ with given moments $\E[X^k]$ for $k=1,\ldots,m$, then sharp bounds on $F_X(x_0) = \Prob(X \leq x_0)$ and on $F_X(x_0-) = \Prob(X < x_0)$ are attained by discrete random variables supported on at most $m+1$ points.
\end{coro}
\begin{proof}
We treat $\bar{F}_m(x_0; c_{1:m})$ first. Define $\bar G_{x_0}:\cPab \to \R$ by $\bar G_{x_0}(\mu) := \mu([a,x_0])$. Since $[a,x_0]$ is relatively closed in $[a,b]$, for $\mu_k \wc \mu$ we have
\[
    \limsup_{k \to \infty} \mu_k([a,x_0]) 
    = \limsup_{k \to \infty} \Prob(X_k \in [a,x_0])
    \leq \Prob(X \in [a,x_0]) 
    = \mu([a,x_0]),
\]
where $X_k \sim \mu_k$, $X \sim \mu$, and the inequality is due to the Portmanteau lemma. Thus the functional $\bar G_{x_0}$ is upper semicontinuous, and by \cref{thm:finitetruncation}, $\bar{F}_m(x_0; c_{1:m}) = \sup_{\mu \in \cFmeps} \bar G_{x_0}(\mu)$ is attained by a measure in $\cPab$ supported on at most $m+1$ points.

Similarly, define $\underline G_{x_0}:\cPab \to \R$ by $\underline G_{x_0}(\mu) := \mu([a,x_0))$. Since $[a,x_0)$ is relatively open in $[a,b]$, for $\mu_k \wc \mu$ we have
\[
    \liminf_{k \to \infty} \mu_k([a,x_0))
    = \liminf_{k \to \infty} \Prob(X_k \in [a,x_0))
    \geq \Prob(X \in [a,x_0))
    = \mu([a,x_0))
\]
again by the Portmanteau lemma. Therefore $\underline G_{x_0}$ is lower semicontinuous and $\underline F_m(x_0;c_{1:m}) = \inf_{\mu \in \cFmeps} \underline G_{x_0}(\mu)$ is again attained by some $\mu^\star\in \cFmeps$, which is again supported on at most $m+1$ points by \cref{thm:finitetruncation}.
\end{proof}

Since the upper and lower bounds in \cref{cor:cdfbounding} are ultimately functions of the first $m$ moments of $X$ (by \cref{cor:momentstoeps}), it is reasonable to check whether they are consistent, in that they should collapse to $F_X(x_0)$ and $F_X(x_0-)$ respectively as $m \to \infty$. \cref{prop:CDFboundsconverge} confirms that this is indeed so:

\begin{prop}\label{prop:CDFboundsconverge}
Under the setup of \cref{cor:cdfbounding}, we have 
\begin{equation*}
    \bar{F}_m(x_0; c_{1:m}) \searrow F(x_0) \quad \text{and} \quad \underline{F}_m(x_0; c_{1:m}) \nearrow F(x_0-)
\end{equation*} 
as $m \to \infty$ for any $x_0 \in [a,b]$. In particular, if $F$ is continuous at $x_0$, then both sequences converge to $F(x_0)$.
\end{prop}
\begin{proof}
    Since $\cF_{m+1}(c_{1:(m+1)}) \subset \cF_{m}(c_{1:m})$, the sequence $\{\bar{F}_m(x_0; c_{1:m})\}_m$ is nonincreasing and the sequence $\{\underline{F}_m(x_0; c_{1:m})\}_m$ is nondecreasing. Moreover, $\mu \in \cF_{m}(c_{1:m})$ for all $m \geq 1$, and so
    \begin{equation}\label{eq:Fseqineq1}
        \underline{F}_m(x_0; c_{1:m})
        \leq \mu([a,x_0)) = F(x_0-) 
        \leq F(x_0) = \mu([a,x_0]) \leq \bar{F}_m(x_0; c_{1:m}), \qquad m \geq 1.
    \end{equation}
    Hence, both sequences have limits by the monotone convergence theorem for real sequences.
    
    We consider $\bar{F}_m(x_0; c_{1:m})$. Equip $\cPab$ with the weak topology. By \cref{cor:cdfbounding}, for each $m$ there exists $\bar{\mu}_m \in \cF_{m}(c_{1:m})$ such that $\bar{\mu}_m([a,x_0]) = \bar{F}_m(x_0; c_{1:m})$. Since $\cPab$ is compact, some subsequence $\{\bar{\mu}_{m_k}\}_k$ converges weakly to a limit $\mu_* \in \cPab$ as $k \to \infty$. Now for any $j \geq 1$, we have $m_k \geq j$ for sufficiently large $k$, so that $\bar{\mu}_{m_k} \in \cF_{m_k}(c_{1:m_k})$ implies $\eps_j(\bar{\mu}_{m_k}) = \eps_j(\mu)$. Because $t \mapsto (b-t)^j \in C([a,b])$, the Portmanteau lemma yields $\eps_j(\mu_*) = \lim_{k \to \infty} \eps_j(\bar{\mu}_{m_k}) = \eps_j(\mu)$. Since $j$ was arbitrary, \cref{thm:Thomeo} guarantees that $\mu_* = \mu$. Now, again by the Portmanteau lemma, we have
    \begin{equation}\label{eq:Fseqineq2}
        \limsup_{k \to \infty} \bar{F}_{m_k}(x_0; c_{1:m_k})
        = \limsup_{k \to \infty} \bar{\mu}_{m_k}([a,x_0])
        \leq \mu([a,x_0]) = F(x_0).
    \end{equation}
    Taking liminfs as $k \to \infty$ in \cref{eq:Fseqineq1} along the subsequence $\{m_k\}_k$ gives 
    \[F(x_0) \leq \liminf_{k \to \infty} \bar{F}_{m_k}(x_0; c_{1:{m_k}})\] and combining this with \cref{eq:Fseqineq2}, we find that
    \begin{equation*}
        \limsup_{k \to \infty} \bar{F}_{m_k}(x_0; c_{1:m_k})
        \leq F(x_0)
        \leq \liminf_{k \to \infty} \bar{F}_{m_k}(x_0; c_{1:{m_k}}),
    \end{equation*}
    from which $\lim_{m \to \infty} \bar{F}_m(x_0; c_{1:m}) \searrow F(x_0)$ follows. $\underline{F}_m(x_0; c_{1:m})$ is treated analogously, which completes the proof.
\end{proof}

The extremizers guaranteed by \cref{cor:momentmatching,cor:cdfbounding} admit nonlinear parameterizations in terms of finitely many support points and weights. In the case of \cref{cor:cdfbounding}, for example,
\begin{equation}\label{eq:NLPmax}
\begin{aligned}
    \bar{F}_m(x_0; c_{1:m})
    &= \max_{\substack{\bx \in [a,b]^{m+1}\\ \bw \geq \bzero}} \bs(\bx; x_0)^\top \bw \\
    &\hspace{0.25in}\text{subject to } 
       \Vb(\bx)^\top \bw = \br \quad \text{and} \quad \bw^\top \bone = 1,\\
\end{aligned}
\end{equation}
where 
\begin{align*}
\br &:= \begin{pmatrix}
    1! c_1, \ldots, m! c_m \\
\end{pmatrix}^\top \in [0,\infty)^m, \\
\bs(\bx; x_0) &:= \begin{pmatrix}
    \Ind{x_1 \leq x_0}, \ldots, \Ind{x_{m+1} \leq x_0}
\end{pmatrix}^\top \in \{0,1\}^{m+1},
\end{align*}
and
\begin{equation*}
    \Vb(\bx) := \begin{bmatrix}
        (b-x_1)^1 & \cdots & (b-x_1)^m \\
        \vdots & \ddots & \vdots \\
        (b-x_{m+1})^1 & \cdots & (b-x_{m+1})^m \\
    \end{bmatrix} \in \R^{(m+1) \times m}.
\end{equation*}
The sharp lower bound for the left limit admits a similar representation:
\begin{equation}\label{eq:NLPmin}
\begin{aligned}
    \underline{F}_m(x_0; c_{1:m}) 
    &= \min_{\substack{\bx \in [a,b]^{m+1}\\ \bw \in  [0,1]^{m+1}}} \bs'(\bx; x_0)^\top \bw \\
    &\hspace{0.25in}\text{subject to } 
   \Vb(\bx)^\top \bw = \br \quad \text{and} \quad \bw^\top \bone = 1,\\
\end{aligned}
\end{equation}
where
\begin{equation*}
    \bs'(\bx; x_0) := \begin{pmatrix}
    \Ind{x_1 < x_0}, \ldots, \Ind{x_{m+1} < x_0}
\end{pmatrix}^\top \in \{0,1\}^{m+1},
\end{equation*}

The upper and lower bounds in \cref{cor:cdfbounding} are more conveniently computed through dual formulations in terms of polynomials, as explained in the context of the generalized duality principle in Chapter~IX of \cite{kreinMarkovMomentProblem1977}. This approach reformulates the search over probability measures as a semi-infinite linear program over the coefficients of a majorizing (or minorizing) polynomial. To compute $\bar{F}_m(x_0; c_{1:m})$, we optimize over polynomials $p^+_m$ of degree at most $m$ with respect to the basis $\{(b-x)^0, \ldots, (b-x)^m\}$ --- say $p^+_m(x) := \sum_{j=0}^m \alpha_j (b-x)^j$ --- subject to the pointwise constraint $p^+_m(x) \geq \Ind{x \leq x_0}$ for all $x \in [a,b]$. Among all such polynomials, we choose the one that minimizes the objective $\sum_{j=0}^m \alpha_j j! c_j$; this objective is linear in the coefficients $\alpha_0, \ldots, \alpha_m$ while the constraints range over $x \in [a,b]$, and so the problem is amenable to a basic cutting-plane method. Analogously, $\underline{F}_m(x_0; c_{1:m})$ is computed by maximizing $\sum_{j=0}^m \beta_j j! c_j$ subject to $\sum_{j=0}^m \beta_j (b-x)^j \leq \Ind{x < x_0}$ for all $x \in [a,b]$. Numerically, the inequality constraints can be enforced over a sufficiently fine grid in $[a,b]$, yielding a standard linear program that can be solved using the simplex method, and the grid can then be iteratively refined by checking and adding violated constraints. A broad array of more advanced methods are also available for such optimization problems \citep[see, e.g.,][]{gobernaLinearSemiInfiniteOptimization1998}.

\begin{ex}
    To demonstrate, suppose we wish to bound $F_m(1/2)$, where $F_m$ is the cdf of a random variable supported on $[0,1]$ whose first $m$ moments coincide with those of the $\stdunif$ distribution. Finding upper and lower bounds on $F_m(1/2)$ amounts to solving \cref{eq:NLPmax,eq:NLPmin} with $c_j = 1/(j+1)!$ for $j=1,\ldots,m$ (see \cref{ex:beta} of \cref{sec:examples}) and $[a,b] = [0,1]$. We derive an analytical solution when $m = 3$. Let $X$ be our target random variable and let $Y = 1-X$. Matching the first $3$ moments of the $\stdunif$ distribution means that $j!c_j = \E[Y^j] = 1/(j+1)$ for $j = 1,2,3$. Write $g^+(y):= \Ind{y \geq 1/2}$ and $g^-(y) := \Ind{y > 1/2}$, corresponding respectively to $\Ind{x \leq 1/2}$ and $\Ind{x < 1/2}$ under the change of variables $y = 1-x$. Now, suppose $p^+_3$ is a polynomial of degree at most $3$ that majorizes $g^+$ on $[0,1]$. Then 
    \begin{equation*}
        F_3\left(\frac{1}{2}\right) = \E[g^+(Y)] \leq \E[p^+_3(Y)] = \int_0^1 p^+_3(y) \dif y
    \end{equation*}
    and analogously if $p^-_3$ minorizes $g^-$ , then
    \begin{equation*}
        F_3\left(\frac{1}{2}-\right) = \Prob\left(X < \frac{1}{2}\right) = \Prob\left(Y > \frac{1}{2}\right) = \E[g^-(Y)] \geq \E[p^-_3(Y)] = \int_0^1 p^-_3(y) \dif y.
    \end{equation*}
    For example, $p^+_3(y) := 4y - 5y^2 + 2y^3 \geq g^+(y)$ on $[0,1]$, and so $F_3(1/2) \leq \int_0^1 p^+_3(y) \dif y = 5/6$. Similarly, $p^-_3(y) := -4y + 11y^2 - 6y^3 \leq g^-(y)$ on $[0,1]$ and $\int_0^1 p^-_3(y) \dif y = 1/6$. We thus conclude that 
    \begin{equation}\label{eq:U01bounds}
        \frac{1}{6} \leq F_3\left(\frac{1}{2}-\right) \leq F_3\left(\frac{1}{2}\right) \leq \frac{5}{6}.
    \end{equation}
    The left and right bounds in \cref{eq:U01bounds} are sharp; indeed, they are attained by the cdf of the discrete random variable $X^*$ that satisfies $\Prob(X^* = 0) = \Prob(X^* = 1) = 1/6$ and $\Prob(X^* = 1/2) = 2/3$, which implies that $p_3^+$ and $p_3^-$ are, in fact, optimal. Closed-form analytic solutions are generally unavailable for $m \geq 4$, so we solve for the bounds numerically. \cref{fig:U01envelope} shows the upper and lower bounds for $m = 1, \ldots, 100$. As guaranteed by \cref{prop:CDFboundsconverge}, the bounds converge to $F_U(1/2) = 1/2$ as $m \to \infty$, where $U \sim \stdunif$.

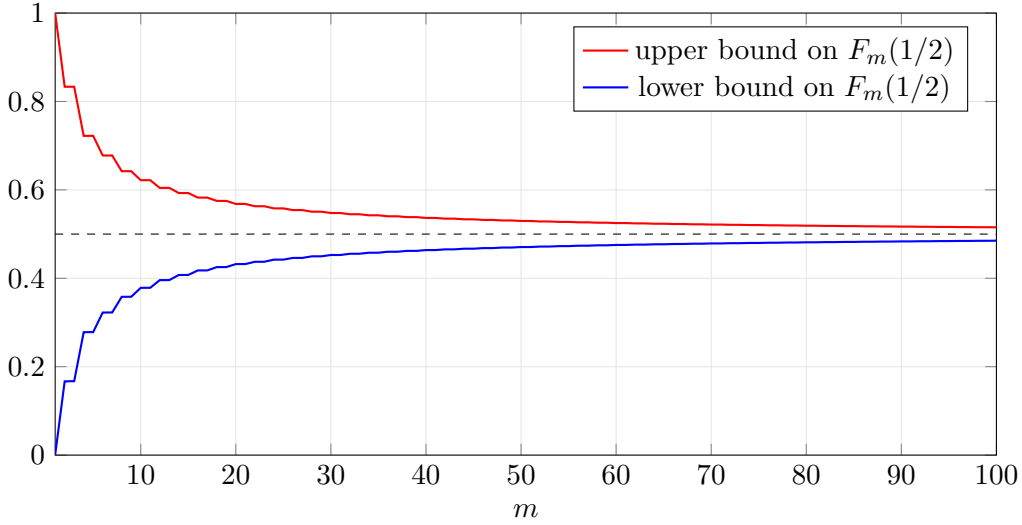
\begin{figure}
        \centering
        \begin{tikzpicture}
            \begin{axis}[
                width=0.85\linewidth,
                height=0.45\linewidth,
                xmin=1, xmax=100,
                ymin=0, ymax=1,
                xlabel={$m$},
                legend pos=north east,
                grid=both,
                major grid style={gray!20},
                minor grid style={gray!10},
            ]

                \addplot+[thick, red, mark=none] table {plot_data/U01_envelope_upper.dat};
                \addlegendentry{upper bound on $F_m(1/2)$}

                \addplot+[thick, blue, mark=none] table {plot_data/U01_envelope_lower.dat};
                \addlegendentry{lower bound on $F_m(1/2)$}

                \addplot+[black, dashed, mark=none] coordinates {(1,0.5) (100,0.5)};
            \end{axis}
        \end{tikzpicture}
        \caption{Sharp upper and lower bounds for $F_m(1/2)$, where $F_m$ is supported on $[0,1]$ and matches the first $m$ moments of the $\stdunif$ distribution, for $m = 1, \ldots, 100$.}
        \label{fig:U01envelope}
    \end{figure}
\end{ex}

\section{Discussion}\label{sec:discussion}

In this paper, we develop an initial theory of the primitive sequence associated with a probability measure supported on a compact interval. In particular, the connection between primitive sequences and Hausdorff moment sequences allows classical moment theoretical ideas to be transported directly into the coordinates defined by the primitive sequence. 

Among the avenues for future work, it is of interest to define primitive sequences for random variables with unbounded support. One natural route for this is to work with primitives ``anchored'' at a given $c \in \R$, of the form $\eps^{(c)}_n := \E[(c-X)_+^n]/n!$, which directly connects the construction to stop-loss transforms and yields a family of distribution-characterizing sequences indexed by $c$. A second route works directly with the mgf $z \mapsto \E[e^{z(b-X)}]$ and interprets $\epsseq$ as the Taylor coefficients of this mgf centered at $z = 0$, provided that certain tail conditions on $X$ are imposed to ensure convergence around $0$. In this setting, the associated questions of admissibility are governed no longer by the Hausdorff moment problem but instead by the Hamburger or Stieltjes moment problems as applicable, and it would be interesting to understand which aspects of the theory for compact intervals survive in that setting, especially for truncated sequences.

It would also be of interest to explore multivariate analogues of the primitive sequence. On $\prod_{h=1}^d [a_h,b_h] \subset \R^d$, a direct analogue of our construction performs iterated integration of the joint cdf coordinatewise and evaluates it at the corner $(b_1, \ldots, b_d)$, producing a tensor $\{\varepsilon_{\bn}\}_{\bn}$ indexed by multi-indices $\bn = (n_1, \ldots, n_d) \in \N^d$, with the representation
\[
    \eps_{\bn} = \frac{1}{\prod_{h=1}^d n_h!}\, \E\left[\prod_{h=1}^d (b_h-X_h)^{n_h}\right].
\]
However, since the geometry of admissible tensors is likely to be far more complex than in one dimension, the choice of truncation family (e.g., total degree versus coordinate-wise truncation) and the structure of extreme points under finitely many constraints would need to be addressed carefully; the multidimensional truncated moment problem \citep[][Part~IV]{schmudgenMomentProblem2017} may provide useful insights into this question.

Finally, a continuous analogue of the primitive sequence can be developed by replacing the $n$-fold iterated integral used to define $\eps_n$ by the Riemann--Liouville fractional integral of order $\alpha>0$, defining
\[
    \eps(\alpha) := \frac{1}{\Gamma(\alpha)} \int_a^b (b-x)^{\alpha - 1} F_X(x) \dif x, \qquad \alpha > 0.
\]
The use of this definition will allow us to generalize the results of \cref{sec:reptheorem}, where $\alpha$ will replace $n$ and $\Gamma(\alpha + 1)$ will replace $n!$. For example, the continuous analogue of \cref{thm:representation} is
\[
    \eps(\alpha) = \frac{1}{\Gamma(\alpha + 1)} \E[(b-X)^\alpha], \qquad \alpha > 0,
\]
and that of \cref{cor:momentstoeps} is
\begin{equation*}
    \eps(\alpha) = \sum_{j=0}^\infty (-1)^j \frac{b^{\alpha-j}}{j! \Gamma(\alpha - j+1)} \E[X^j]
    \quad \text{and} \quad
    \E[X^\alpha] = \sum_{j=0}^\infty \frac{\Gamma(\alpha+1)}{\Gamma(\alpha-j+1)} b^{\alpha-j}\eps(j), \qquad \alpha > 0.
\end{equation*}
A thorough study of the consequences and potential of this generalization is left for future work.

\section{Acknowledgments}
Primitive sequences --- then called ``canonical primitives'' --- were defined for the special case $[a,b] = [0,1]$ in Section~2.3 of \cite{zimmermanCopulasNewTheory2025}, where \cref{cor:momentstoeps,cor:seqboundsequality} were given with different proofs. The second identity in \cref{eq:momentstoeps} of \cref{cor:momentstoeps} also appeared as Lemma~2 of \cite{lalancetteTractableFamilySmooth2025} for the same special case.

\bibliographystyle{apalike}
\bibliography{References_2026}

\end{document}